\documentclass{article}

\usepackage{PRIMEarxiv}

\usepackage[utf8]{inputenc} 
\usepackage[T1]{fontenc}    
\usepackage{hyperref}       
\usepackage{url}            
\usepackage{booktabs}       
\usepackage{amsfonts}       
\usepackage{nicefrac}       
\usepackage{microtype}      
\usepackage{lipsum}
\usepackage{fancyhdr}       
\usepackage{graphicx}       
\usepackage{amsmath}
\usepackage{amsthm}
\graphicspath{{media/}}     

\newtheorem{theorem}{Teorema}[section]

\title{Mathematical Model of the Impact of Chemotherapy and Anti-Angiogenic Therapy on Drug Resistance in Glioma Growth}

\author{
 Latifah Hanum \\
  Universitas Gadjah Mada\\
  Yogyakarta, Indonesia \\
  \texttt{latiifahhanum@mail.ugm.ac.id} \\
   \And
 Nanang Susyanto \\
  Universitas Gadjah Mada\\
  Yogyakarta, Indonesia \\
  \texttt{nanang\_susyanto@ugm.ac.id} \\
  \And
 Dwi Ertiningsih \\
  Universitas Gadjah Mada\\
  Yogyakarta, Indonesia \\
  \texttt{dwi\_ertiningsih@ugm.ac.id} \\
}

\begin{document}
\maketitle

\begin{abstract}
This research presents a mathematical model of glioma growth dynamics with drug resistance, capturing interactions among five cell populations: glial cells, sensitive glioma cells, resistant glioma cells, endothelial cells, and neuron cells, along with two therapy agent populations: chemotherapy and anti-angiogenic therapy. Glioma is a malignant tumor originating from glial cells, undergoes chemotherapy-induced mutations, leading to drug-resistant glioma cells. This not only impacts glioma cells but also normal cells. Combining chemotherapy and anti-angiogenic therapy, the model employs a Holling type II response function, considering optimal dosages for treatment optimization. Through analysis, three equilibrium are identified: two stable and one unstable equilibrium points. Numerical simulations, employing phase portraits and trajectory diagrams, illustrate the combined therapies impact on glioma cells. In summary, this concise model explores glioma dynamics and drug resistance, offering insights into the efficacy of combined therapies, crucial for optimizing glioma treatment.
\end{abstract}


\section{Introduction}
Tumors cells are abnormal cells that are classified as either benign or malignant. Benign tumors exhibit characteristics that do not invade normal tissues, while malignant tumors can invade and spread throughout the entire body. Glioma is a type of aggressive brain tumor that originates from glial cells. From an epidemiological standpoint, gliomas occur across all age groups, but they are more commonly observed in adults, with males being more susceptible than females \cite{yang2022}. Solid tumors, such as gliomas in the brain, that grow larger than a critical size (1-2 $mm$ in diameter) need to recruit new blood vessels to supply the necessary oxygen and required nutrients for their survival and growth. This process involves the formation of new blood vessels \cite{cea2012}.

Blood vessels within gliomas are utilized for delivering nutrients and facilitating the migration of cancer cells. Glioma cells migrate alongside blood vessels, displacing the interactions between glial cells and blood vessels. Through this mechanism, glioma cells manage to extract nutrients from the bloodstream. This relocation process disrupts the function of glial cells, compromising the proper delivery of sufficient glucose and oxygen to neurons \cite{ye1999}. These consequences carry broader significance, impacting the well-being of neurons that rely on the support provided by glial cells for their nourishment, structural stability \cite{gless1955}, and the maintenance of the extracellular environment that envelops these neurons. Within this intricate neural network, neurons assume a pivotal role in the processing of sensory and internal signals, thereby making a substantial contribution to our cognitive and perceptual abilities. Malignant gliomas are characterized by significant blood vessel enhancement due to the process of angiogenesis, which plays a crucial role in tumor growth and colonization within the brain. The blood vessels in gliomas demonstrate the proliferation of endothelial cells, indicating a high-grade glioma \cite{wur2009}.

Despite these advancements, a promising option for cancer treatment is chemotherapy. Currently, 9 out of 10 chemotherapy failures are linked to drug resistance. In chemotherapy, after administering specific drugs, a significant number of patient tumor cells become resistant to those drugs. Consequently, drug resistance emerges as a serious issue in the field of cancer. Drug resistance often imposes inevitable limitations on the long-term effectiveness of therapies targeted at cancer patients. Significant efforts have been made to combat drug resistance and enhance patient survival. Although the underlying molecular and cellular mechanisms are intricate, several paradigmatic mechanisms of drug resistance have been established. It is widely accepted that the inherent heterogeneity within the population of cancer cells is believed to encompass cells that are sensitive to drugs as well as cells that are resistant to drugs \cite{man2017}.

One of the combination therapies in chemotherapy is anti-angiogenic therapy. Anti-angiogenic therapy is a method used to combat cancer with the aim of cutting off the supply of nutrients and oxygen to tumor cells through the blood vessels and preventing the formation of new blood vessels. Most approved anti-angiogenic agents for cancer treatment rely on targeting the \textit{Vascular Endothelial Growth Factor} (VEGF), as VEGF signaling is considered a primary promoter of angiogenesis. In addition to controlling angiogenesis, these drugs can enhance immunotherapy because VEGF also exhibits immunosuppressive functions, highlighting a potential target of anti-angiogenic therapy. Targeting blood vessels in brain tumors has become an extremely intriguing strategy, considering the high rate of endothelial proliferation, vascular permeability, and expression of proangiogenic growth factors \cite{coel2021}.

The author's interest in similar research stemmed from the realization that the process of angiogenesis can influence the response of tumor cell growth. This insight was then applied to cases of glioma tumors with drug resistance, which are highly malignant primary brain cancers. The model the author devised takes the form of a system of differential equations. This model depicts the dynamics of glioma brain tumor cell growth in the presence of drug resistance. The drug resistance, as a result of chemotherapy, renders tumor cells insensitive to chemotherapy drugs. The role of endothelial cells in ongoing therapy, particularly anti-angiogenic therapy, becomes crucial as they play a pivotal role in delivering supplies from anti-angiogenic agents. This has a significant impact on cancer since the required blood flow for tumor cells can be disrupted, leading to cell dormancy. The Dormancy phase describes the state where tumor cells cease to grow and become inactive for a certain period. As a result, the transition of tumor cells into dormancy is a mechanism that facilitates the survival and development of tumors. Therefore, gaining a better understanding of the proportional dynamics of resistant cells within gliomas under chemotherapy can lead to further therapeutic approaches. This approach holds the potential for effectiveness because by cutting off the blood supply, tumor cells cannot grow or develop. Consequently, the created model introduces additional compartments and new parameters to analyze their influence. 

This paper is organized into the following sections. In the next section, we'll introduce our model formulation. Section 3 discusses existence and stability of the equilibria. Section 4 will present numerical simulations to back up our theoretical findings. Finally, we'll wrap things up in Section 5 with our conclusions and discussions.

\section{Model Formulation}
\label{sec2}
Previous research on glioma growth models has been conducted by \cite{tro2020} is divided into five classes: glial cells, sensitive glioma cells, resistant glioma cells, neurons, and chemotherapy agents. Another journal introduced by \cite{pinho2012} discusses a model related to anti-angiogenic therapy for tumors in general. The author then attempts to create a combination of anti-angiogenic therapy and chemotherapy to observe their impact on gliomas in the presence of drug resistance, based on scientific research by \cite{wur2009}. The development performed in this study lies in the effect of subsequent therapy that influences the glioma tumor cell compartments. The crucial impact of angiogenesis on the growth of glioma brain tumors is referenced from the study by \cite{boc2018}, thus adding parameters related to dormancy in the tumor cell compartments as an effect of anti-angiogenic therapy. In the formulation of this mathematical model, a nonlinear system of differential equations is defined to represent the mathematical impact of anti-angiogenic therapy and chemotherapy on glioma tumors with drug resistance. 

The cells population is separated into five parts, i.e., the glial cells ($G_1$), the glioma sensitive cell ($G_2$), the glioma resistance cell ($G_3$), the endothelial cell ($G_4$), the neuron cell ($G_5$) and two agents of treatment are agent chemotherapy ($Q$), and agent anti-angiogenic ($Y$). The schematic interaction diagram which shows the interaction between the normal cells, tumour cells and agent therapy has been shown on Figure \ref{fig1}.

\begin{figure}
\centering
\includegraphics[width=6in]{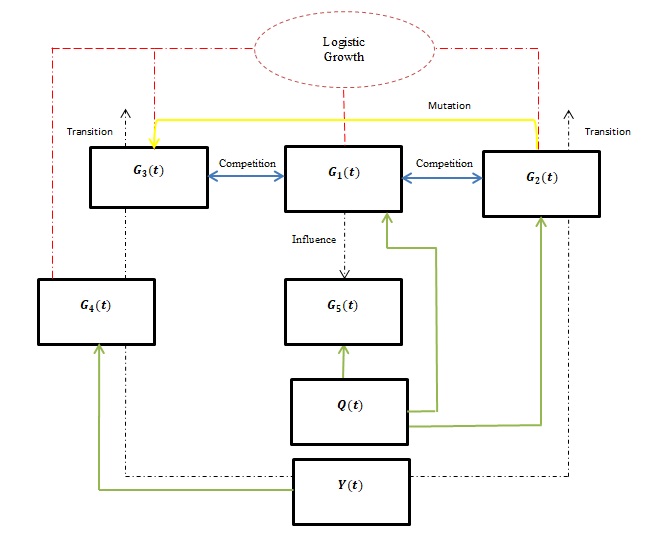}
\caption{The interaction diagram of glial cells, sensitive glioma cells, resistant glioma cells, endothelial cells, neurons, chemotherapy agents, and anti-angiogenic agents. }
\label{fig1}
\end{figure}

From the diagram, we have the system as follows:
\begin{eqnarray}
    \label{1}
    && \frac{d G_1}{d t}=p_1 G_1\left[1-\frac{G_1}{k_1}\right]-\kappa_1 G_1\left[G_2+G_3\right]-D_1\left(G_4, Y\right) \frac{G_1 Q}{A_1+G_1} \nonumber \\ 
    && \frac{d G_2}{d t}=p_2 G_2\left[1-\frac{G_2+G_3}{k_2+\chi G_4}\right]-\kappa_2 G_1 G_2-u F(Q) G_2-\rho F(Y) G_2 \nonumber \\
    &&\;\;\;\;\;\;\;\;\;\;\;\;-D_2\left(G_4, Y\right) \frac{G_2 Q}{A_2+G_2} \nonumber \\
    && \frac{d G_3}{d t}=p_3 G_3\left[1-\frac{G_2+G_3}{k_2+\chi G_4}\right]-\kappa_3 G_1 G_3+u F(Q) G_2 \nonumber \\
    &&\;\;\;\;\;\;\;\;\;\;\;\;-\rho F\left(Y\right) G_3 \\
    && \frac{d G_4}{d t}=\Phi\left[G_2+G_3\right]+p_4 G_4\left[1-\frac{G_4}{k_3}\right]-D_4 \frac{G_4 Y}{A_4+G_4} \nonumber \\
    && \frac{d G_5}{d t}=\omega\dot{G}_1 F\left(-\frac{\dot{G_1}}{k_1}\right) G_5-D_5\left(G_4, Y\right) \frac{G_5 Q}{A_5+G_5} \nonumber\\
    && \frac{d Q}{d t}=\phi-\left[\psi+c_1 \frac{G_1}{A_1+G_1}+c_2 +\frac{G_2}{A_2+G_2}+c_5\frac{G_5}{A_5+G_5}\right] Q \nonumber \\
    && \frac{d Y}{d t}=\delta-\left[\gamma+c_4 \frac{ G_4}{A_4+G_4}\right] Y \nonumber
\end{eqnarray}
with
\begin{eqnarray*}
    D_i(G_4,Y)=D_{i0}+D_{i1} G_4+D_{i2} Y , i=1,2,5
\end{eqnarray*}
and initial values $G_i\geq0,i=1,...,5$ $Q\geq0,Y\geq0$ for $t=0.$
The Heaviside function is generally defined as follows:
\begin{equation}
F(x)=
\begin{cases}
0, & x \leq 0 \\
1, & x > 0.
\end{cases}
\label{eq:custom-equation}
\end{equation}

 Variable $G_1$ is the glial cells concentration ($kg.m^{-3}$), $G_2$ is the drug-sensitive glioma cells concentration ($kg.m^{-3}$), $G_3$ is the drug-resistant glioma cells concentration ($kg.m^{-3}$), $G_4$ is the endothelial cells concentration ($kg.m^{-3}$), $G_5$ is the neurons concentration ($kg.m^{-3}$), $Q$ is the chemotherapeutic agent concentration ($mg.m^{-2}$), $Y$ is the anti-angiogenic agent concentration ($mg.m^{-2}$).
Parameters $p_i, i=1,2,3,4$ represent proliferation rate following logistic growth with carrying capacity $k_i, i=1,2,3,4$. Parameters ($\kappa_i, i=1,2,3$) are competition coefficient between glial and glioma cells. Parameter $\chi$ is the proportion of endothel cells responsible for the glioma angiogenesis. Parameter $u$ is mutation rate of sensitive glioma to resistent glioma. Parameter $\rho$ is transition rate of glioma cells into dormancy phase triggered by anti-angiogenic therapy. Parameter $\Phi$ is the rate of creation of glioma cells due to endothel cells. Neuron loss results from reduced glial concentration, while an increase has no effect, depicted by $\omega$. Parameters $D_{i 0}, i=1,2,5$, is the killing rate of chemotherapy on $X_i$ in the absence of $G_4$ and $Y$ respectively. Parameters $D_{i j}, i, j=1,2,5$, is the rate of increased killing on $G_i$ by chemotherapy agent per concentration of $G_4(j=1)$ and $Y(j=2)$. Parameter $D_4$ is the killing rate of anti-angiogenic therapy on $G_4$. 
$A_i, i=1,2,4,5$, is the Holling type II constant for $G_i$. Exponential increase in chemotherapy agent concentration due to infusion at rate $\phi$, followed by decrease related to body clearance of chemotherapy sent to glioma patients at clearance rate $\psi$. Exponential increase in anti-angiogenic agent concentration due to infusion at rate $\delta$, followed by decrease related to body clearance of anti-angiogenic sent to glioma patients at clearance rate $\gamma$. Parameters $c_i, i=1,2,4,5$, is the rate at which agent anti-angiogenic and chemotherapy agent combine with $G_i$. 

We obtain the non-dimensionalized model according to the procedure outlined in \cite{jia2020}, if we let  $g1=G_i/k_i, i=1,2,3,4, q=Q, y=Y, \beta_1=\kappa_1 k_2, \beta_2=\kappa_2 k_1, \beta_3=\kappa_3 k_1, \alpha=\omega k_1, a_i=A_i/k_i (i=1,2,4,5)$, $\mu=\Phi k_3/k_2, t=\chi k_3/k_2, d_{i0}=D_{i0}/k_i, d_{i1}=D_i k_3/k_i, d_{i2}=D_{i2}/k_i, (i=1,2,5), d_4=D_4/k_3,$ then our model becomes:
\begin{eqnarray}
    \label{3}
    && \frac{d g_1}{d t}=p_1 g_1[1-g_1]-\beta_1 g_1\left[g_2+g_3\right]-d_1\left(g_4, y\right) \frac{g_1 q}{a_1+g_1} \nonumber \\ 
    && \frac{d g_2}{d t}=p_2 g_2\left[1-\frac{g_2+g_3}{1+\tau g_4}\right]-\beta_2 g_1 g_2-u F\left(q\right) g_2 \nonumber \\
    &&\;\;\;\;\;\;\;\;\;\;\;\;-\rho F\left(y\right) g_2-d_2\left(g_4, y\right) \frac{g_2 q}{a_2+g_2} \nonumber \\
    && \frac{d g_3}{d t}=p_3 g_3\left[1-\frac{g_2+g_3}{1+\tau g_4}\right]-\beta_3 g_1 g_3+u F(q) g_2 \nonumber \\
    &&\;\;\;\;\;\;\;\;\;\;\;\;-\rho F\left(y\right) g_3  \\
    && \frac{d g_4}{d t}=\mu\left[g_2+g_3\right]+p_4 g_4 [1-g_4]-d_4\frac{ g_4 y}{a_4+g_4} \nonumber \\
    && \frac{d g_5}{d t}=\alpha \dot{g}_1 F\left(- \dot{g_1}\right) g_5-d_5\left(g_4, y\right) \frac{g_5 q}{a_5+g_5} \nonumber\\
    && \frac{d q}{d t}=\phi-\left[\psi+c_1 \frac{g_1}{a_1+g_1}+c_2 \frac{g_2}{a_2+g_2}+c_5 \frac{g_5}{a_5+g_5}\right] q \nonumber \\
    && \frac{d y}{d t}=\delta-\left[\gamma+c_4 \frac{ g_4}{a_4+g_4}\right] y \nonumber
\end{eqnarray}
with
\begin{eqnarray*}
    d_i(g_4,y)=d_{i0}+d_{i1} g_4+d_{i2} y , i=1,2,5
\end{eqnarray*}
and initial values $g_i\geq0,i=1,...,5$ $q\geq0,y\geq0$ for $t=0.$

In the following section, we will explore the characteristics of the solution for the model (\ref{3}), encompassing its existence and local stability of the equilibria.

\section{Existence And Local Stability Of Equilibria}
\label{sec:3}
\subsection{The Existence of Equilibria}
The equilibrium point of the glioma model, referred to as $E_0=(0,0,0,0,0,\frac{\phi}{\varphi},\frac{\delta}{\gamma})$, shows the nonexistence of the all cells. This state at which the system's variables remain constant over time.  

 The others equilibria
\begin{eqnarray*}
E_1=(g_1^b,0,0,g_4^b,0,q^b,y^b)
\end{eqnarray*}
where 
\begin{eqnarray*}
    q^b&=&\frac{p_1[1-g_1^b][a_1+g_1^b]}{d_{10}+d_{11} g_4^b+d_{12} y^b} \label{xx}\\
    y^b&=&\frac{\delta[a_4+g_4^b]}{a_4 \gamma+c_4 g_4^b+g_4^b \gamma} \label{x}
\end{eqnarray*}
and furthermore, to determine the value of $g_4^b$ and $g_1^b$ when $g_2=0$ and $g_3=0$,
we obtain at least one positive root for the quadratic equation: 
\begin{equation}
    p_4(c_4+\gamma) g_4^2+p_4((a_4-1) \gamma-c_4) g_4-a_4 \gamma p_4+\delta d_4=0 
\end{equation}
by applying Descarte's rule of signs, 
\begin{equation}
g_4^b=\frac{\left(-\gamma a_4+\gamma+c_4\right)+\sqrt{\left(\gamma a_4-\gamma-c_4\right)^2-4\left(\gamma+c_4\right)\left[\delta\left(d_4\right) / p_4-a_4 \gamma \right]}}{2\left(\gamma+c_4\right)}
\end{equation}
for the existence of the equilibriua $E_1$ note that if $\delta\left(d_4\right)<a_4 \gamma p_4$, implying $-\gamma a_3+\gamma+c_4 < \left(\gamma a_4-\gamma-c_4\right)^2 - 4\left(\gamma+c_4\right)\left[\delta\left(d_4\right) / p_4-a_4 \gamma \right]$.
Next, by simple calculation we have at least one positive root  for the
quadratic equation:
\begin{equation} \label{6}
     p_1(c_1+\psi) g_1^{2}+((a_1-1) \psi-c_1) p_1 g_1-a_1 p_1 \psi+\phi d_1=0 
 \end{equation}
then descarte's rule of signs yields a quadratic Equation (\ref{6}) is given by:
    \begin{equation} \label{aww}
g_1^b=\frac{\left(\psi+c_1-\psi a_1\right)+\left[\left(-\psi-c_1+\psi a_1\right)\right]^2-4\left(\psi+c_1\right)\left[d_1\left(g_4, y\right) \phi/ p_1-\psi a_1\right]^{1 / 2}}{2\left(\psi+c_1\right)}.
\end{equation}
Therefore, if $\phi/\psi<p_1 a_1/(d_{10}+d_{11} g_4^b+d_{12} y^b)$, $g_1^b>0$ always exists, as well as for $q^b$.

Next, resistant glioma equilibria represent  where tumor cells still persist within the host individual affected by glioma. In other words, $g_2 \neq 0$ or $g_3 \neq 0$. When affected by glioma, equilibrium points in the population occur in two conditions: first, when $g_2>0$ and $g_3=0$, if $g_2=0$, leading to the values $g_2=0$ and $g_3=0$ when the equilibrium point is a free glioma equilibrium point. Thus, we will investigate the condition where $g_2=0$ and $g_3>0$. Since $g_3\neq 0$,  we found equilibrium $E_2=(0,0,g_3^r,g_4^r,0,q^r,y^r)$ where:
\begin{eqnarray*}
       g_3^r&=&\frac{(g_4^r \tau + 1)(p_3-\rho)}{p_3}\\
       q^r&=&\frac{\phi}{\psi}\\
        y^r &=&  \frac{{\delta  (a_4 + g_4^r)}}{{(c_4 + \gamma)  g_4^r + a_4  \gamma}}
\end{eqnarray*}
and $g_4^r$ is positive solutions of the following equation:
\begin{eqnarray}\label{gogo}
l_1 g_4^3 + l_2 g_4^2 + l_3 g_4 + l_4 = 0
\end{eqnarray}
where $l_i$, for $i=1,2,3,4$, are defined as follows:
\begin{eqnarray*}
l_1 &=& 1\\
l_2 &=& \frac{((-\mu \tau + (a_4 - 1) p_4) p_3 + \mu \rho \tau)}{p_3 p_4}\\
l_3 &=& \frac{(((-a_4 \tau - 1) \mu + d_4 y - a_4 p_4) p_3 + \rho \mu (a_4 \tau + 1))}{p_3 p_4}\\
l_4 &=& \frac{a_4 \mu (\rho - p_3)}{p_3 p_4}.
\end{eqnarray*}
The solutions of
equation (\ref{gogo}) are involves the following steps as provided by [19], resulting in roots:
\begin{eqnarray*}
    g_{4,1}&=\frac{\sqrt[3]{A+\sqrt{A^2+4 B^3}}}{3 \sqrt[3]{2}}-\frac{\sqrt[3]{2} B}{3 \sqrt[3]{A^2+\sqrt{A^2+4 B^3}}}-\frac{l_2}{3}\\
    g_{4,2}&=-\frac{(1-i \sqrt{3}) \sqrt[3]{A+\sqrt{A^2+4 B^3}}}{6 \sqrt[3]{2}}+\frac{(1+i \sqrt{3}) \sqrt[3]{2} B}{6 \sqrt[3]{A+\sqrt{A^2+4 B^3}}}-\frac{l_2}{3}\\
    g_{4,3}&=-\frac{(1+i \sqrt{3}) \sqrt[3]{A+\sqrt{A^2+B}}}{6 \sqrt[3]{2}}+\frac{(1-i \sqrt{3}) \sqrt[3]{2} B}{6 \sqrt[3]{A+\sqrt{A^2+B}}}-\frac{l_2}{3}
\end{eqnarray*}
with
\begin{align*}
    A&=9 l_2 l_3-27 l_4-2 l_2^3 \\
    B&=3 l_3-l_2^2.
\end{align*}

Next, we will analyze the real roots of the polynomial and the conditions under which the real roots of the polynomial are negative. Let's assume:
$$
P=\sqrt[3]{A+\sqrt{A^2+4 B^3}}.
$$
The condition for $P$ to be real is $A^2 \geq B$. With this, we have:
\begin{eqnarray*}
     g_{4,1}&=&\frac{P}{3 \sqrt[3]{2}}-\frac{\sqrt[3]{2} B}{3 P}-\frac{l_2}{3}\\
    g_{4,2}&=&-\frac{(1-i \sqrt{3}) P}{6 \sqrt[3]{2}}+\frac{(1+i \sqrt{3}) \sqrt[3]{2} B}{6 P}-\frac{l_2}{3}\\
    g_{4,3}&=&-\frac{(1+i \sqrt{3}) P}{6 \sqrt[3]{2}}+\frac{(1-i \sqrt{3}) \sqrt[3]{2} B}{6 P}-\frac{l_2}{3}.
\end{eqnarray*}
The equation $g_4^r$ is one positive real root. For this root to be positive, the condition $\frac{P}{3 \sqrt[3]{2}}>\frac{\sqrt[3]{2} B}{3 P}+\frac{l_2}{3}$ must be satisfied. 

\subsection{The Stability of equilibria}
In this section, we're going to investigate the stability of the equilibria by examining the eigenvalues of the Jacobian matrices in system (\ref{3}). We can determine stability by looking at whether the real parts of the eigenvalues of the Jacobian matrix, calculated at a specific equilibrium point, are positive or negative.  

For the local stability of glioma-free equilibrium point $E_0, E_1$, anda glioma resistent equilibrium point $E_2$, we have the following theorem.
\begin{theorem}
    The first glioma-free equilibrium point, denoted as $E_0 = (0, 0, 0, 0, 0, \frac{\phi}{\varphi}, \frac{\delta}{\gamma})$, exists and is locally asymptotically stable if the following conditions are: $\phi>\frac{p_1 \psi a_1 \gamma}{(d_{10} + d_{12} \delta)}$, $\phi>\frac{(p_2-u-\rho)\psi a_2 \gamma }{(d_{20}+d_{22} \delta)}$, $\rho>p_3$, $\delta>\frac{p_4 \gamma a_4}{d_4}$. These conditions ensure the stability of the equilibrium point $E_0$.
\end{theorem}
\begin{proof}
To determine the stability of $E_0$ one may compute the variational matrix of system (\ref{3}) about $E_0$ given by
\begin{eqnarray*}
D\mathbf{f}{E_0} &=& \left[\begin{array}{ccccccc}
m_{11} & 0 & 0 & 0 & 0 & 0 & 0 \\
0 & m_{22} & 0 & 0 & 0 & 0 & 0 \\
0 & u & m_{33} & 0 & 0 & 0 & 0 \\
0 & \mu & \mu & m_{44} & 0 & 0 & 0 \\
0 & 0 & 0 & 0 & m_{55} & 0 & 0 \\
-\frac{{c_1\phi}}{{\psi a_1}} & -\frac{{c_2\phi}}{{\psi a_2}} & 0 & 0 & -\frac{{c_5\phi}}{{\psi a_5}} & m_{66} & 0 \\
0 & 0 & 0 & -\frac{{c_4\delta}}{{\gamma a_4}} & 0 & 0 & m_{77}
\end{array}\right]. \\
\end{eqnarray*}

The eigenvalues of the Jacobian matrix of system (\ref{3}) at equilibrium point $E_0 = (0, 0, 0, 0, 0, \frac{\phi}{\varphi}, \frac{\delta}{\gamma})$ are as follows:

\begin{eqnarray*}
&&\lambda_1 = p_1 - \frac{{(d_{12}\delta/\gamma + d_{10})\phi}}{{\psi a_1}} \\
&&\lambda_2 = p_2 - u -\rho - \frac{{(d_{22}\delta/\gamma + d_{20})\phi}}{{\psi a_2}} \\
&&\lambda_3 = p_3 - \rho \\
&&\lambda_4 = p_4 - \frac{{d_4\delta}}{{\gamma a_4}} \\
&&\lambda_5 = -\frac{{(d_{52}\delta/\gamma + d_{50})\phi}}{{\psi a_5}} \\
&&\lambda_6 = -\psi \\
&&\lambda_7 = -\gamma. \label{4.41}
\end{eqnarray*}
Furthermore, it is necessary that $\phi>\frac{p_1 \psi a_1 \gamma}{(d_{10} + d_{12} \delta)}$, $\phi>\frac{(p_2-u-\rho)\psi a_2 \gamma }{(d_{20}+d_{22} \delta)}$, $\rho>p_3$, and $\delta>\frac{p_4 \gamma a_4}{d_4}$ where these results are obtained through $\lambda_1<0$, $\lambda_2<0$, $\lambda_3<0$ and $\lambda_4<0$. Given that the parameters $\phi$, $\psi$, and $\gamma$ are positive, we know that $\lambda_5$, $\lambda_6$, and $\lambda_7$ are all negative. We consider $p_1=0.0068$, $p_2=0.012$, $p_3=0.002$, $p_4=0.002$, $\psi=0.01813$, $\gamma=0.136$, $\rho=0.01$, $d_{10}=4.7\times10^{-8}$, $d_{12}=3.9\times10^{-8}$, $d_{20}=7.8\times10^{-2}$, $d_{22}=7.5$, $d_4=0.71$, $u=0.01$, and $a_1=a_2=a_4=1$ (Table). With these values, we obtain that $E_0$ is linearly
asymptotically stable for $\phi>356.22$. In other words, the chemotherapy agent kills all cells, They will never recuperate. However, the non-cell state's stability is only ensured at an exceptionally high infusion rate $\phi$. Consequently, the equilibrium point $E_0$ is unstable.
\end{proof}

\begin{theorem}
    If the following conditions are satisfied:
\begin{enumerate}
    \item $2p_{4}g_{4}^b\gamma + \frac{2p_{4}g_{4}^{2b}c_{4}}{a_{4} + g_{4}^b} + \frac{d_{4}y_{a_{4}}^b\gamma}{(a_{4} + g_{4}^b)^{2}} > p_{4}\gamma + \frac{p_{4}c_{4}g_{4}^b}{a_{4} + g_{4}^b}$
    \item $d_2(g_4^b,y^b) q^b>(p_2 - \beta_2 g_1^b - u - \rho) a_2$
    \item $p_3<\beta_3 g_1^b  - \rho$
    \item $2 \psi p_1 g_1^b + \frac{2 c_1 g_1^{2b} p_1}{a_1 + g_1^b} + \frac{a_1 d_{11} g_4^b q^b (a_1 \psi + c_1 g_1^b + g_1^b \psi)}{(a_1 + g_1^b)^3}+\frac{a_1 d_{12} q^b y^b (a_1 \psi + c_1 g_1^b + g_1^b \psi)}{(a_1 + g_1^b)^3} + \frac{a_1 d_{10} q^b (a_1 \psi + c_1 g_1^b + g_1^b \psi)}{(a_1 + g_1^b)^3} > \psi p_1 +\frac{c_1 g_1^b p_1}{a_1 + g_1^b}$
\end{enumerate}
Then, the second glioma-free equilibrium point $E_1=(g_1^b,0,0,g_4^b,0,q^b,y^b)$ in system (\ref{3}) is locally asymptotically stable.
\end{theorem}
\begin{proof}
    The stability of $E_1$ can be determined by analyzing the eigenvalues of the Jacobian Matrix, which is the result of linearizing system (\ref{3}) around the equilibrium point $E_1=(g_1^b,0,0,g_4^b,0,q^b,y^b)$.
  \begin{equation}
D\mathbf{f}{E_1} = \left[
\begin{array}{ccccccc}
e_{11} & e_{12} & e_{13} & e_{14} & 0 & e_{16} & e_{17} \\
0 & e_{22} & 0 & 0 & 0 & 0 & 0 \\
0 & e_{32} & e_{33} & 0 & 0 & 0 & 0 \\
0 & e_{42} & e_{43} & e_{44} & 0 & 0 & e_{47} \\
0 & 0 & 0 & 0 & e_{55} & 0 & 0 \\
e_{61} & e_{62} & 0 & 0 & e_{65} & e_{66} & 0 \\
0 & 0 & 0 & e_{74} & 0 & 0 & e_{77}
\end{array}
\right].
\end{equation}
    The obtained eigenvalues are as follows: 
\begin{eqnarray*}
    \lambda_1&=&p_2 - \beta_2 g_1^b - u - \rho-\frac{d_2(g_4^b,y^b) q^b}{a_2} \\
    \lambda_2&=&p_3-\beta_3 g_1^b  - \rho\\
    \lambda_3&=&-\frac{d_5(g_4^b,y^b) q^b}{a_5}.
\end{eqnarray*}
 Parameters $d_5,a_5>0$ so $\lambda_3<0$. In order to ensure the stability of $E_1=(g_1^b,0,0,g_4^b,0,q^b,y^b)$ it is necessary that
$d_2(g_4^b,y^b) q^b>(p_2 - \beta_2 g_1^b - u - \rho) a_2$ and $p_3<\beta_3 g_1^b  - \rho$, where these results are obtained through $\lambda_1<0$ and $\lambda_2<0$. The values of the dimensionless parameters are positives, then the eigenvalues  $\lambda_1$, $\lambda_2$ are negatives.

Next, for $\lambda_4$ and $\lambda_5$, we obtain the following:

\begin{eqnarray}\label{uu}
    \lambda^2 - v_1\lambda + v_2
\end{eqnarray}
where:
\begin{eqnarray*}
    v_1 &=& e_{77} + e_{44}\\
    v_2 &=& e_{77}e_{44} - e_{74}e_{47}\\
    &=& \left(2 \gamma p_4 g_4^b - \gamma p_4 + \frac{\gamma d_4 y^b a_4}{(a_4 + g_4^b)^2} + \frac{2 c_4 g_4^{2b} p_4}{a_4 + g_4^b} - \frac{c_4 g_4^b p_4}{a_4 + g_4^b}\right)\\
    &&-\left(\frac{c_4 a_4 y^b d_4 g_4^b}{(a_4 + g_4^b)^3}\right)\\
    &=& 2 \gamma p_4 g_4^b - \gamma p_4 + \frac{\gamma d_4 y^b a_4}{(a_4 + g_4^b)^2} + \frac{2 c_4 g_4^{2b} p_4}{a_4 + g_4^b} - \frac{c_4 g_4^b p_4}{a_4 + g_4^b}.
\end{eqnarray*}
Based on the known values, $e_{74} < 0$ and $e_{47} < 0$, so $e_{74} e_{47} > 0$. Then, if 
\begin{eqnarray*}
2p_{4}g_{4}\gamma + \frac{2p_{4}g_{4}^{2}c_{4}}{a_{4} + g_{4}} + \frac{d_{4}y_{a_{4}}\gamma}{(a_{4} + g_{4})^{2}} > p_{4}\gamma + \frac{p_{4}c_{4}g_{4}}{a_{4} + g_{4}},
\end{eqnarray*}
then $e_{77}e_{44} > 0$. Furthermore, we have:
\begin{eqnarray*}
&&2 \gamma p_4 g_4 - \gamma p_4 + \frac{\gamma d_4 y a_4}{(a_4 + g_4)^2} + \frac{2 c_4 g_4^2 p_4}{a_4 + g_4} - \frac{c_4 g_4 p_4}{a_4 + g_4} + \frac{c_4 a_4 y d_4 g_4}{(a_4 + g_4)^3}> \frac{c_4 a_4 y d_4 g_4}{(a_4 + g_4)^3}
\end{eqnarray*}
so, $e_{77}e_{44} > e_{74}e_{47}$. 
Now, from Equation (\ref{uu}), we obtain the eigenvalues $\lambda_{4,5}$ as follows:
$$
\lambda_{4,5} = \frac{-v_1 \pm \sqrt{v_1^2 - 4(1)v_2}}{2(1)} = \frac{-v_1 \pm \sqrt{v_1^2 - 4 v_2}}{2}.
$$

The real part of $\lambda_4$ and $\lambda_5$ is always negative because if $v_1^2 - 4 v_2 > 0$, then $v_1^2 > v_1^2 - 4 v_2$, which implies:

$$
\begin{aligned}
& \lambda_4 = \frac{-v_1 + \sqrt{v_1^2 - 4 v_2}}{2} < 0 \\
& \lambda_5 = \frac{-v_1 - \sqrt{v_1^2 - 4 v_2}}{2} < 0 .
\end{aligned}
$$
Additionally, if $v_1^2 < 4 v_2$, then $\sqrt{v_1^2 - 4 v_2}$ is imaginary, and since $v_1 > 0$, $-v_1 < 0$. Moving on to $\lambda_6$ and $\lambda_7$, we have:
\begin{eqnarray}\label{ux}
    \lambda^2 - w_1\lambda + w_2
\end{eqnarray}
where:
\begin{eqnarray*}
    w_1 &=& e_{66} + e_{11}\\
    w_2 &=& e_{66}e_{11} - e_{16}e_{61}\\
    &=& \left(2 \psi p_1 g_1^b + \frac{2 c_1 {g_1^b}^2 p_1}{a_1 + g_1^b} - \psi p_1 - \frac{c_1 g_1^b p_1}{a_1 + g_1^b} + \frac{a_1 d_{11} g_4^b q^b (a_1 \psi + c_1 g_1^b + g_1^b \psi)}{(a_1 + g_1^b)^3}\right.\\
    &&+ \left.\frac{a_1 d_{12} q^b y^b (a_1 \psi + c_1 g_1^b + g_1^b \psi)}{(a_1 + g_1^b)^3} + \frac{a_1 d_{10} q^b (a_1 \psi + c_1 g_1^b + g_1^b \psi)}{(a_1 + g_1^b)^3}\right)\\
    &&- \left(\frac{(d_{11} g_4^b + d_{12} y^b + d_{10}) g_1^b c_1 a_1 q^b}{(a_1 + g_1^b)^3}\right)\\
    &=& \left(2 \psi p_1 g_1^b + \frac{2 c_1 {g_1^b}^2 p_1}{a_1 + g_1^b} - \psi p_1 - \frac{c_1 g_1^b p_1}{a_1 + g_1^b} + \frac{q^b a_1 (d_1(g_4^b,y^b)(a_1 \psi + c_1 g_1^b + g_1^b \psi))}{(a_1 + g_1^b)^3}\right)\\
    &&- \left(\frac{(d_1(g_4^b,y^b) g_1^b c_1 a_1 q^b}{(a_1 + g_1^b)^3}\right)\\
    &=& 2 \psi p_1 g_1^b + \frac{2 c_1 {g_1^b}^2 p_1}{a_1 + g_1^b} - \psi p_1 - \frac{c_1 g_1^b p_1}{a_1 + g_1^b} + \frac{q^b a_1(d_1(g_4^b,y^b)(a_1 \psi + g_1^b \psi))}{(a_1 + g_1^b)^3}.
\end{eqnarray*}

Based on the known values, $e_{16}$ and $e_{61}$ are negative, so $e_{16}e_{61} > 0$. If
\begin{eqnarray*}
&&2 \psi p_1 g_1^b + \frac{2 c_1 g_1^{2b} p_1}{a_1 + g_1^b} + \frac{q^b a_1 (d_1(g_4^b,y^b)(a_1 \psi + c_1 g_1^b + g_1^b \psi))}{(a_1 + g_1^b)^3} > \psi p_1 + \frac{c_1 g_1^b p_1}{a_1 + g_1^b},
\end{eqnarray*}

then $e_{66}e_{11} > 0$. Furthermore, we have:
\begin{eqnarray*}
&&2 \psi p_1 g_1^b + \frac{2 c_1 {g_1^b}^2 p_1}{a_1 + g_1^b} - \psi p_1 - \frac{c_1 g_1^b p_1}{a_1 + g_1^b} + \frac{q^b a_1 (d_1(g_4^b,y^b)(a_1 \psi + c_1 g_1^b + g_1^b \psi)}{(a_1 + g_1^b)^3}> \frac{(d_1(g_4^b,y^b) g_1^b c_1 a_1 q^b}{(a_1 + g_1^b)^3}
\end{eqnarray*}
So, $e_{66}e_{11} > e_{16}e_{61}$. 
Now, from Equation (\ref{ux}), we obtain the eigenvalues $\lambda_{6,7}$ as follows:

$$
\lambda_{6,7} = \frac{-w_1 \pm \sqrt{w_1^2 - 4(1)w_2}}{2(1)} = \frac{-w_1 \pm \sqrt{w_1^2 - 4 w_2}}{2}.
$$
The real part of $\lambda_6$ and $\lambda_7$ is always negative because if $w_1^2 - 4 w_2 > 0$, then $w_1^2 > w_1^2 - 4 w_2$, which implies:

$$
\begin{aligned}
& \lambda_6 = \frac{-w_1 + \sqrt{w_1^2 - 4 w_2}}{2} < 0 \\
& \lambda_7 = \frac{-w_1 - \sqrt{w_1^2 - 4 w_2}}{2} < 0 .
\end{aligned}
$$
Additionally, if $w_1^2 < 4 w_2$, then $\sqrt{w_1^2 - 4 w_2}$ is imaginary, and since $w_1 > 0$, $-w_1 < 0$. Because $\lambda_i < 0$ for every $i = 1, 2, 3, 4, 5, 6, 7$, the equilibrium point $E_1$ is locally asymptotically stable.
Therefore, the stability of the second free equilibrium point of glioma signifies a state in which glioma cells no longer exist within the body of the glioma patient. This indicates that chemotherapy and anti-angiogenic therapy at specific dosages can suppress glioma growth while ensuring that all neuronal cells are spared from any adverse effects.
\end{proof}
In the next section we will give some numerical simulation
to illustrate the theoretical results for several case.

\section{Numerical Result}
In this section, we will showcase several numerical simulations aimed at illustrating the theoretical findings discussed in the preceding section. The parameter values employed in the system, as detailed in Table (\ref{table_2}), were acquired from the work of \cite{tro2020,pinho2012}.
\begin{table}[!t]
\renewcommand{\arraystretch}{1.3}
\caption{Parameters of the model.}
\label{table_2}
\centering
\begin{tabular}{|c||c||c|}
\hline
\textbf{Parameter} & \textbf{Value} & \textbf{Reference} \\
\hline
$ p_1 $ & $ 0.0068 \text{ day}^{-1} $ & $ p_1<p_2 $ \cite{pinho2012}  \\
\hline
$ p_2 $ & $ 0.012 \text{ day}^{-1} $ & \cite{tro2020} \\
\hline
$ p_3 $ & $ 0.002 \text{ day}^{-1} $ & \cite{tro2020} \\
\hline
$ p_4 $ & $ 0.002 \text{ day}^{-1} $ & $ p_4<p_1 $ \cite{alb2002} \\
\hline
$ k_i,i=1,2,3,4 $ & $ 510 \text{ kg.m}^{-3} $ & \cite{tro2020}  \\
\hline
$ A_i,i=1,2,4,5 $ & $ 510 \text{ kg.m}^{-3} $ & \cite{tro2020}  \\
\hline
$ \kappa_1 $ & $ 3.6 \times 10^{-5} \text{ day}^{-1} $ & \cite{pinho2012}
 \\
\hline
$ \kappa_2,\kappa_3 $ & $ 3.6 \times 10^{-6} \text{ day}^{-1} $ & $ \kappa_2,\kappa_3 <\kappa_1 $ \cite{pinho2012}  \\
\hline
$ \chi $ & $ 0.15  $ & $ \chi<1 $ \cite{pinho2012} \\
\hline
$ \Phi $ & $ 0.004 \text{ day}^{-1} $ & $ \Phi>p_4 $ \cite{sa2001} \\
\hline
$ c_1 $ & $ 0.0002 \text{ day}^{-1} $ & \cite{pinho2012}
 \\
\hline
$ c_2 $ & $ 0.032 \text{ day}^{-1} $ & $ c_2 \gg c_1 $ \cite{pinho2012}  \\
\hline
$ c_4 $ & $ 0.032 \text{ day}^{-1} $ & $ c_4 \geq c_2 $ \cite{hana1996} \\
\hline
$ c_5 $ & $ 0.0012 \text{ day}^{-1} $ & $ c_5 \geq c_1 $ \\
\hline
$ u $ & $ 0-1 $ & \cite{tro2020} \\
\hline
$ \rho $ & $ 0-1 $ & \cite{boc2018} \\
\hline
$ \omega $ & $0-0.02  $ & \cite{tro2020} \\
\hline
$ D_{10} $ & $ 2.4 \times 10^{-5} \text{ day}^{-1} $ &\cite{pinho2012}
 \\
\hline
$ D_{20} $ & $ 4.0 \times 10 \text{ day}^{-1} $ & $ D_{20} \gg D_{10} $ \cite{sil1987}\\
\hline
$ D_{50} $ & $ 2.4  \text{ day}^{-1} $ & $ D_{20} > D_{50} \gg D_{10} $ \\
\hline
$ D_4 $ & $ 3.6\times 10^{2} \text{ day}^{-1} $ & $ D_3 > D_{20} $ \cite{pinho2012} \\
\hline
$ D_{11} $ & $ 4.0 \times 10^{-8} \text{ day}^{-1} $ & $ D_{11} < D_{10} $ \cite{pinho2012} \\
\hline
$ D_{21} $ & $ 4.0 \times 10^{-2} \text{ day}^{-1} $ & $ D_{21} > D_{11} $  \cite{pinho2012}\\
\hline
$ D_{51} $ & $ 4.0 \times 10^{-3} \text{ day}^{-1} $ & $ D_{21} > D_{51} > D_{11} $ \\
\hline
$ D_{12} $ & $ 2.0\times 10^{-5} \text{ day}^{-1} $ & $ D_{12} < D_{10} $ \cite{pinho2012} \\
\hline
$ D_{22} $ & $ 3.8\times10^{3} \text{ day}^{-1} $ & $ D_{22} > D_{12} $ \cite{pinho2012} \\
\hline
$ D_{52} $ & $ 2.0 \text{ day}^{-1} $ & $ D_{22} > D_{52} > D_{12} $ \\
\hline $ \phi $ & $ 3.3 \times 10^{-3} \text{ day}^{-1} $ & Almost continuous \cite{bow2001} \\
\hline $ \psi $ & $ 0.01813 \text{ day}^{-1} $ & Half-life of CA \cite{said2007} \\
\hline $ \delta $ & $ 2.4 \times 10^{-4} \text{ day}^{-1} $ & $ \phi=14 \delta $ \cite{bow2001} \\
\hline $ \gamma $ & $ 0.136  \text{ day}^{-1}$ & $ \gamma=7.5 \psi $ \cite{sus2001} \\
\hline
\end{tabular}
\end{table}

\begin{table}[!t]
\renewcommand{\arraystretch}{1.3}
\caption{Parameters of the non-dimensionalization model.}
\label{table2}
\centering
\begin{tabular}{|c||c|}
\hline
\textbf{Parameter} & \textbf{Value} \\
\hline
$ a_i,i=1,2,4,5 $ & $ 1 $  \\
\hline
$ \beta_1 $ & $ 1.8 \times 10^{-2} $ \\
\hline
$ \beta_2,\beta_3 $ & $ 1.8 \times 10^{-3} $ \\
\hline
$ \tau $ & $ 0.15 $ \\
\hline
$ \mu $ & $ 0.004 $ \\
\hline
$ \alpha $ & $0-10$ \\
\hline
$ d_{10} $ & $ 4.7 \times 10^{-8} $ \\
\hline
$ d_{20} $ & $ 7.8 \times 10^{-2} $ \\
\hline
$ d_{50} $ & $ 4.7 \times 10^{-3} $ \\
\hline
$ d_4 $ & $ 0.71 $ \\
\hline
$ d_{11} $ & $ 4.0 \times 10^{-8} $ \\
\hline
$ d_{21} $ & $ 4.0 \times 10^{-2} $ \\
\hline
$ d_{51} $ & $ 4.0 \times 10^{-3} $ \\
\hline
$ d_{12} $ & $ 3.9 \times 10^{-8} $ \\
\hline
$ d_{22} $ & $ 7.5 $ \\
\hline
$ d_{52} $ & $ 3.9 \times 10^{-3} $ \\
\hline 
\hline
\end{tabular}
\end{table}

\subsection{Numerical Simulation and Biological Interpretation of the Second Glioma Free Equilibrium Points}

In the case of tumor-free conditions, there are no tumor cells within the brain tissue, meaning there are no sensitive glioma cells nor resistant glioma cells ($g_2=0$ and $g_3=0$). Subsequently, a numerical simulation is conducted based on the parameter values in Table \ref{table2} to assess the stability of the free glioma equilibrium point $E_1= (0.99,0,0,0.65,0,0.18,0.0016)$ in the phase portrait and trajectory diagram. Therefore, in the following simulation, the values of $u = 0.001, \alpha=2$ and $\rho = 0.01$ are chosen. 
\begin{figure}[!t]
\centering
\includegraphics[width=4in]{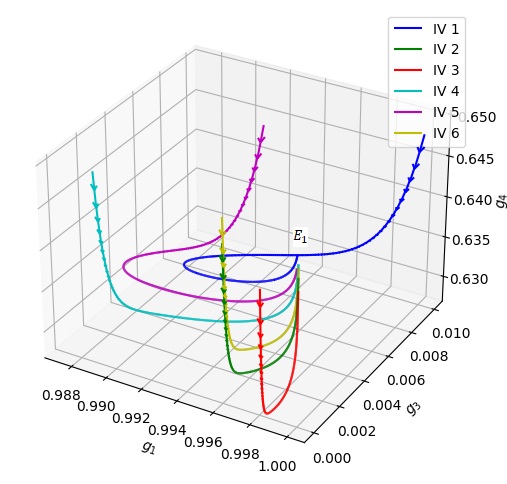}
\caption{Phase portrait of the glioma free equilibrium point $E_1= (0.99,0,0,0.65,0,0.18,0.0016)$}
\label{fig2}
\end{figure}
 Figure \ref{fig2} shows that there exists a single equilibrium point. This equilibrium point is locally asymptotically stable and has a stable \textit{node} shape. When different initial values for the concentrations of glial cells, resistant glioma cells, and endothelial cells are taken, they converge to a certain value. 
The local stability analysis around this equilibrium point is provided by the Jacobian matrix below.
$$
D f\left(E_1\right)\\={\tiny
\left[
\begin{array}{ccccccc}
-0.006 & -0.018 & -0.018 & -3.620 \times 10^{-9} & 0 & -3.667 \times 10^{-8} & -3.549 \times 10^{-9} \\
0 & -0.021 & 0 & 0 & 0 & 0 & 0 \\
0 & 0.001 & -0.009 & 0 & 0 & 0 & 0 \\
0 & 0.004 & 0.004 & -0.001 & 0 & 0 & -0.279 \\
0 & 0 & 0 & 0 & -0.001 & 0 & 0 \\
-9.051 \times 10^{-6} & -0.005 & 0 & 0 & -0.000 & -0.018 & 0 \\
0 & 0 & 0 & -0.000 & 0 & 0 & -0.148 \\
\end{array}
\right]
}
 \text {. }
$$
As a result, the eigenvalues of $D f\left(E_1\right)$ are calculated as follows: $\lambda_1 = -0.001$, $\lambda_2 = -0.149$, $\lambda_3 = -0.001$, $\lambda_4 = -0.009$, $\lambda_5 = -0.022$, $\lambda_6 = -0.018$, and $\lambda_7 = -0.007$. Since the real parts of all $\lambda_i$ are negative ($\operatorname{Re}\left(\lambda_i\right) < 0$), the equilibrium point $E_1 = (0.99, 0, 0, 0.65, 0, 0.18, 0.0016)$ is locally asymptotically stable.

Furthermore, a trajectory diagram of cell concentrations, including glial cells, sensitive glioma cells, resistant glioma cells, endothelial cells, and neurons, concerning initial conditions, is provided. Here, the authors depict the trajectory diagram of equilibrium points for free glioma at $t=0-10000$.
\begin{figure}[!t]
\centering
\includegraphics[width=4in]{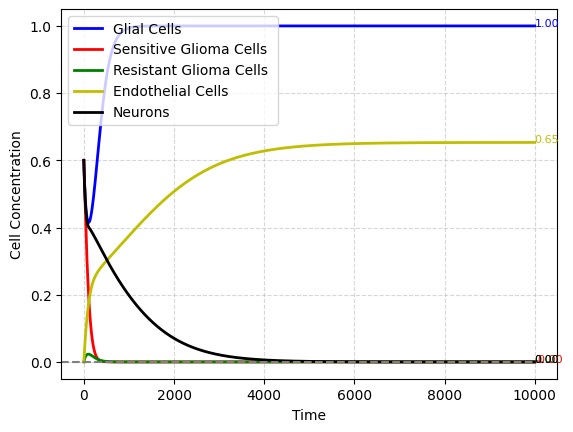}
\caption{Trajectory diagram of free glioma equilibrium points ($t=0-10000$)}
\label{fig_traye1}
\end{figure}

In Figure \ref{fig4}, it can be observed that as time approaches infinity ($t \to \infty$), the concentration of glial cells initially decreases over the course of 100 days. Then, around day 101, it starts to increase until it reaches the value of $g_1=0.99$ or a concentration of glial cells equal to $0.99$, remaining constant thereafter and approaching nearly twice its initial value for an extended period. On the other hand, the concentration of sensitive glioma cells continually decreases from the outset until it approaches zero over the course of 200 days, extending to infinity ($t \to \infty$). Meanwhile, the concentration of resistant glioma cells initially increases, almost reaching a value of $0.2$ on day 75, after which its growth declines on day 76 and eventually approaches zero indefinitely. 

This results in stimulation of endothelial cells. As shown in Figure \ref{fig5}, the concentration of endothelial cells continues to increase every day, reaching a value of $0.65$ for an indefinite period. Subsequently, the concentration of neuron cells decreases in tandem with the decline in glial cells until day 100, after which it remains nearly constant and approaches zero indefinitely.

Therefore, the stable equilibrium point of the second free glioma represents a condition where glioma cells no longer exist in the body of glioma patients. This indicates that chemotherapy and anti-angiogenic therapy with specific dosages can suppress glioma growth, with the side effect of eliminating all neuron cells as well.
\begin{figure}[!t]
\centering
\includegraphics[width=3in]{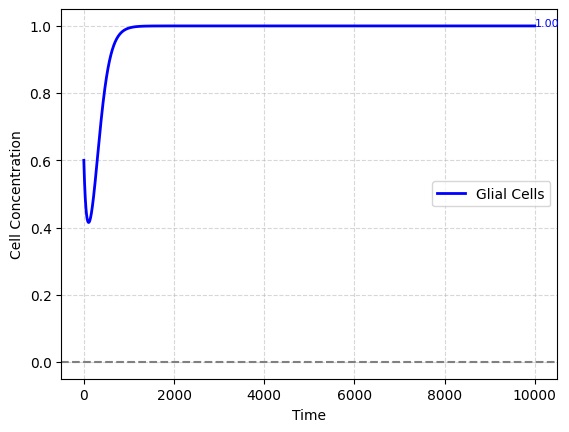}
\includegraphics[width=3in]{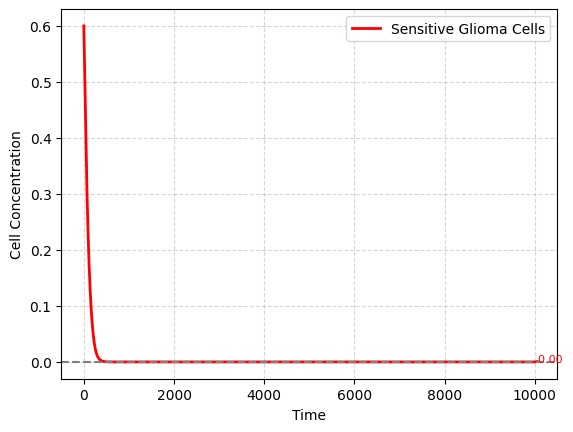}
\includegraphics[width=3in]{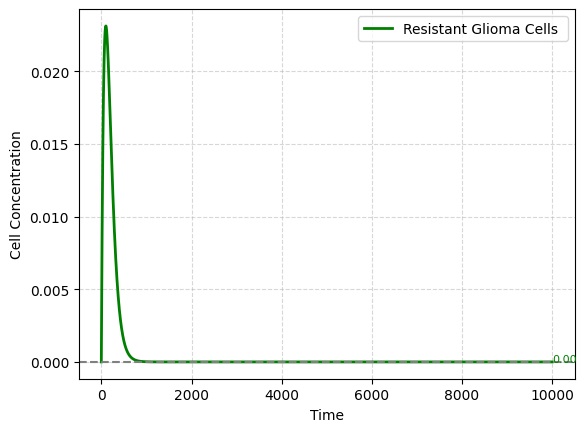}
\caption{Trajectory diagram of free glioma equilibrium points for glial cells, sensitive glioma cells, and resistant glioma cells}
\label{fig4}
\end{figure}

\begin{figure}[!t]
\centering
\includegraphics[width=3in]{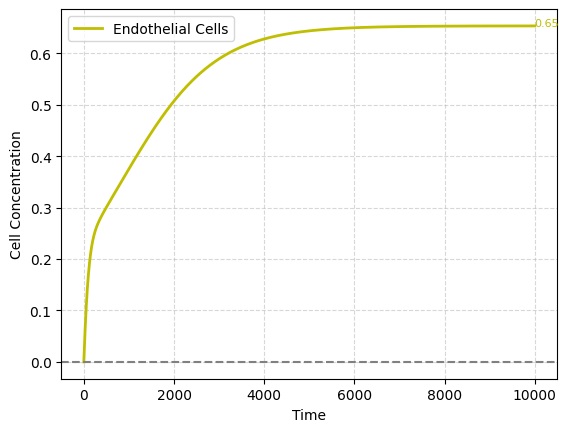}
\includegraphics[width=3in]{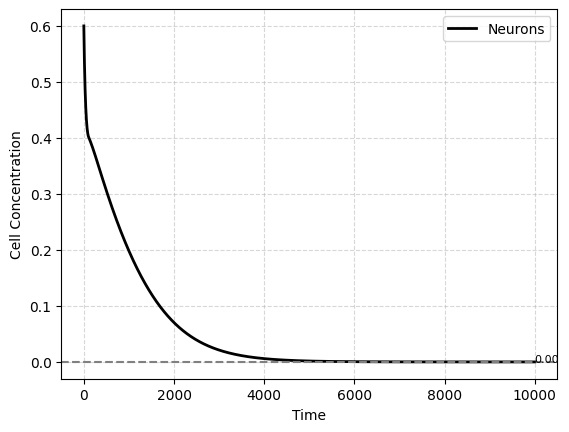}
\caption{Trajectory diagram of free glioma equilibrium points for endothelial cells and neurons}
\label{fig5}
\end{figure}

\subsection{Numerical Simulation and Biological Interpretation of the Equilibrium Point of Resistant Glioma}

In the case of tumor resistance, tumor cells still exist within the brain tissue, meaning that resistant glioma cells persist in the body of glioma patients. Subsequently, a numerical simulation is performed based on the parameter values in Table \ref{table2}, and parameters corresponding to the existence of the equilibrium point of resistant glioma are selected to assess the stability of the equilibrium point of free resistant glioma $E_2=(0,0,g_3^r,g_4^r,0,q^r,y^r)$ with $g_3^r,g_4^r,q^r$, and $y^r$ given in section (\ref{sec:3}).

\begin{figure}[!t]
\centering
\includegraphics[width=4in]{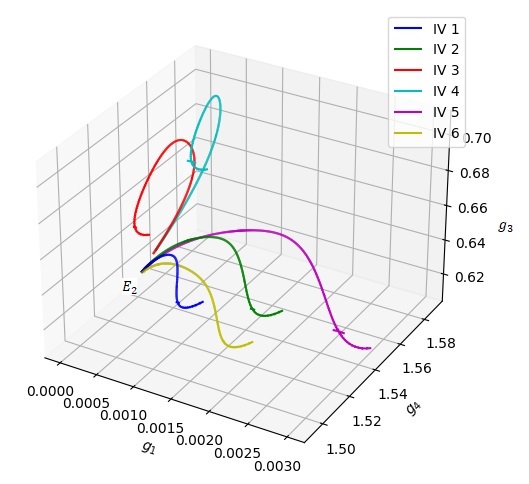}
\caption{Phase Portrait of Resistant Glioma Equilibrium Points $E_2=(0,0,g_3^r,g_4^r,0,q^r,y^r)$}
\label{ppres}
\end{figure}

To obtain an endemic case, we choose the values $p_3 = 0.006$, $u = 0.01$, $\rho = 0.003$, $\phi = 4.0 \times 10^{-3}$, $\alpha=2$ and $\delta = 2.9 \times 10^{-4}$. In Figure \ref{ppres}, it can be shown that there is one equilibrium point that emerges, which is the equilibrium point for resistant glioma, denoted as $E_2 = (0, 0, g_3^r, g_4^r, 0, q^r, y^r)$. This equilibrium point is locally asymptotically stable and has a spiral shape. When different initial values are chosen for the concentrations of glial cells, resistant glioma cells, and endothelial cells, they all converge to a certain value. The local stability analysis around this equilibrium point is provided by the Jacobian matrix below.
$$
D f\left(E_2\right)\\={\tiny
\left[
\begin{array}{ccccccc}
-0.0046 & 0 & 0 & 0 & 0 & 0 & 0 \\
0 & -0.0400 & 0 & 0 & 0 & 0 & 0 \\
-0.0011 & 0.0070 & -0.0030 & 0.0002 & 0 & 0 & 0 \\
0 & 0.0040 & 0.0040 & -0.0044 & 0 & 0 & -0.4280 \\
0 & 0 & 0 & 0 & -0.0024 & 0 & 0 \\
-0.00004 & -0.0070 & 0 & 0 & -0.0003 & -0.01813 & 0 \\
0 & 0 & 0 & -9.256 \times 10^{-6} & 0 & 0 & -0.1554\\
\end{array}
\right]
}
 \text {. }
$$
Subsequently, the eigenvalues were obtained as follows: $\lambda_1=-0.03996$, $\lambda_2=-0.0181$, $\lambda_3=-0.1554 + 2.9 \times 10^{-27}i$, $\lambda_4=-0.0025 + 1.98 \times 10^{-25}i$, $\lambda_5=-0.0048 - 2 \times 10^{-25}i$, $\lambda_6=-0.0024$, and $\lambda_7=-0.0046$. Since each $\operatorname{Re}(\lambda_i)$ for $i=1,2,3,4,5,6,7$ is negative, the equilibrium point $E_2=(0,0,g_3^r,g_4^r,0,q^r,y^r)$ is locally asymptotically stable.

Furthermore, trajectory diagrams of cell concentrations, including glial cells, sensitive glioma cells, resistant glioma cells, endothelial cells, and neurons, concerning initial conditions, are provided. In Figure \ref{trayres}, we shows the trajectory diagram of the equilibrium point of resistant glioma at $t=0-10000$.

\begin{figure}[!t]
\centering
\includegraphics[width=4in]{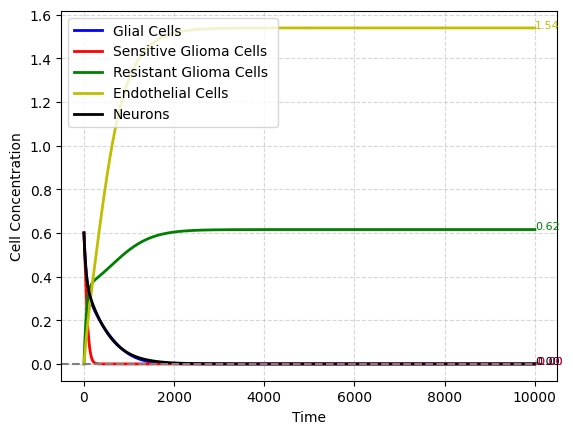}
\caption{ Trajectory diagram of resistant glioma equilibrium points ($t=0-10000$)}
\label{trayres}
\end{figure}

 It is further clarified in Figure \ref{fig_8} that as $t \to \infty$, the variables $g_1$ and $g_2$ approach zero. The concentration of glial cells decreases to nearly zero over an indefinite period. Meanwhile, the concentration of sensitive glioma cells continues to decrease initially until it approaches zero by $t=200$ days and continues to zero for an indefinite time ($t \to \infty$). Furthermore, the concentration of resistant glioma cells increases for 1400 days and subsequently remains constant at $0.62$ in the indefinite time ($t \to \infty$). This is reasonable as it falls within the capacity limits of resistant glioma cell growth.

Figure \ref{fig_9} displays the trajectory diagram for variables $g_4$ and $g_5$. The continuous increase in resistant glioma cells stimulates angiogenesis processes in the brain, resulting in the stimulation of endothelial cell formation, leading to a significantly higher concentration of endothelial cells, exceeding twice the concentration of resistant glioma cells, at $1.54$.

Therefore, the stable equilibrium point of resistant glioma depicts the presence of resistant tumor cells that persist within the brain tissue but does not surpass the threshold for resistant glioma cell growth ($k_2=1$) while considering specific chemotherapy and anti-angiogenic dosages.

\begin{figure}[!t]
\centering
\includegraphics[width=3in]{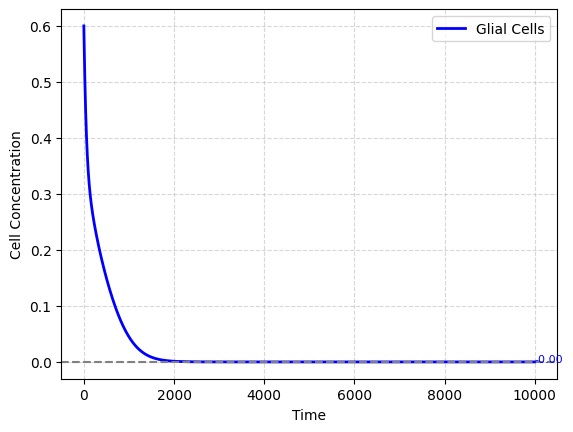}
\includegraphics[width=3in]{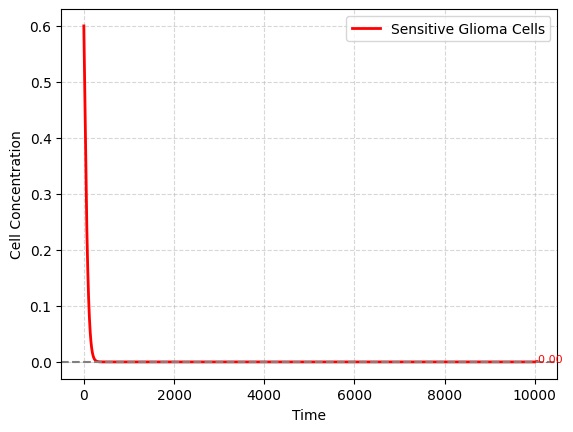}
\includegraphics[width=3in]{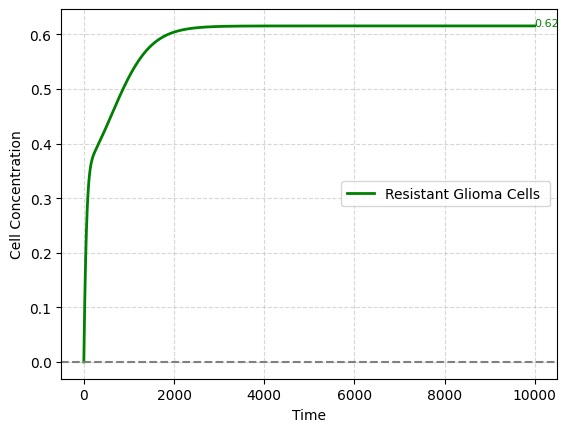}
\caption{Trajectory diagram of resistant glioma equilibrium points for glial cells, sensitive glioma cells, and resistant glioma cells}
\label{fig_8}
\end{figure}

\begin{figure}[!t]
\centering
\includegraphics[width=3in]{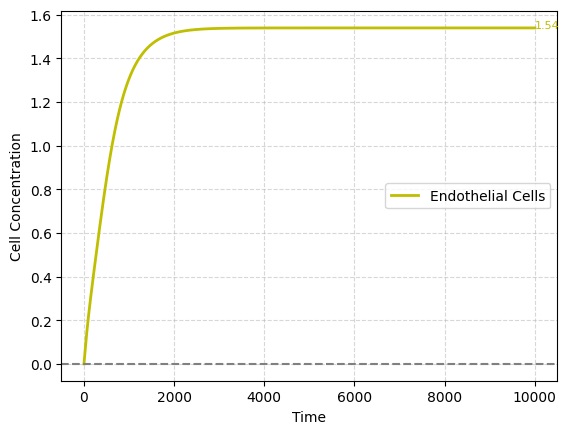}
\includegraphics[width=3in]{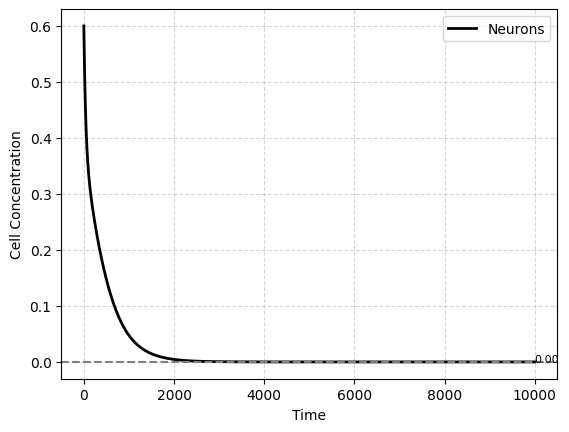}
\caption{Trajectory diagram of resistant glioma equilibrium points for endothelial cells and neurons}
\label{fig_9}
\end{figure}

\subsection{The Behavior of Subpopulations of Sensitive Glioma Cells and Resistant Glioma Cells with Varying Transition Rates to Dormancy Phase}

\begin{figure}[!t]
\centering
\includegraphics[width=4in]{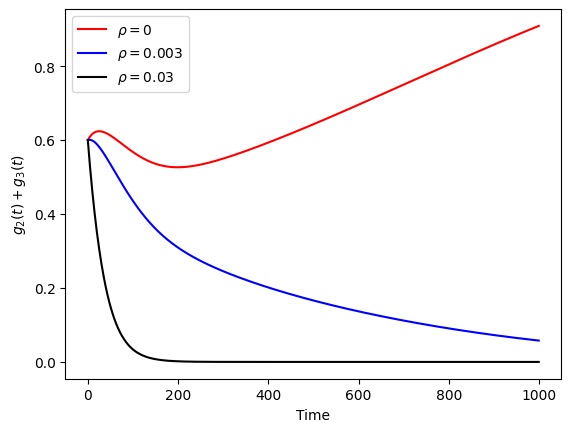}
\caption{ Time evolution of $g_2(t)+g_3(t)$ for various values of $\rho$.}
\label{vber}
\end{figure}
Based on the graph in Figure \ref{vber} above, two cases can be considered:
\begin{itemize}
    \item Case 1: If $\rho \geq p_3$. From Figure \ref{vber}, it can be observed that the concentration of glioma cells gradually decreases when the transition rate of cells to the dormant phase ($\rho$) is greater than or equal to the proliferation rate of resistant glioma cells ($p_3$). This implies the possibility of successful chemotherapy and anti-angiogenic therapy. When $\rho > p_3$, as time $\mathrm{t}$ approaches infinity, the concentration of glioma cells tends to be unstable because the proliferation of resistant glioma cells exceeds the rate of transition to the dormant phase.
    \item Case 2: If $\rho < p_3$. From Figure \ref{vber}, it can be observed that the concentration of glioma cells will increase when the transition rate of cells to the dormant phase ($\rho$) is less than the proliferation rate of resistant glioma cells ($p_3$). This implies the success of chemotherapy and anti-angiogenic therapy in reducing the maximum concentration of glioma cells by $1-\frac{(g_4 \tau + 1)(p_3-\rho)}{p_3}$.
\end{itemize}

\section{Conclusion}
In conclusion, this research paper provides a comprehensive exploration of the multifaceted dynamics surrounding glioma brain tumors and their responsiveness to chemotherapy, with a particular emphasis on the crucial role of anti-angiogenic therapy in the treatment process. The investigation commences by establishing a mathematical model to simulate the intricate interactions taking place within the tumor microenvironment. This model incorporates various factors, such as the diffusion of chemotherapy agents, proliferation of tumor cells, and the influence of anti-angiogenic drugs, making it a valuable tool for assessing the efficacy of treatment strategies. One of the pivotal aspects addressed in this study is the identification and characterization of equilibrium points within the mathematical model. Equilibrium points represent states in which the system remains stable over time, providing critical insights into the potential outcomes of chemotherapy and anti-angiogenic therapy. Through rigorous analysis, it becomes apparent that equilibrium points can vary considerably based on parameter values, specifically infusion rates and the strengths of cellular interactions. This variation highlights the sensitivity of the system to these factors and underscores the need for precise therapeutic planning.

Furthermore, the analysis sheds light on the attainability and stability of equilibrium points that signify tumor eradication. Notably, it is discerned that achieving a tumor-free state often requires exceptionally high infusion rates of chemotherapy agents. This revelation has significant implications for clinical practice, suggesting that complete tumor elimination through chemotherapy alone may be a challenging endeavor due to potential toxicity concerns and limitations in drug delivery systems. In this context, the study underlines the importance of considering a comprehensive treatment approach that combines chemotherapy with anti-angiogenic therapy. Anti-angiogenic drugs, designed to inhibit the formation of blood vessels within tumors, can complement chemotherapy by altering the tumor microenvironment, potentially reducing the required chemotherapy dosage and minimizing side effects. This synergistic approach addresses the issue of drug resistance often encountered in glioma treatment.

In summary, this research underscores the intricate and interdependent nature of glioma treatment. It emphasizes the need for personalized therapeutic strategies that take into account the unique parameters of each patient's tumor and considers the delicate balance between chemotherapy and anti-angiogenic therapy. By offering a deeper understanding of the mathematical foundations governing these treatments, this study contributes valuable insights to the ongoing efforts to improve the prognosis and quality of life for individuals facing the formidable challenge of glioma brain tumors.

\bibliographystyle{unsrt}  
\bibliography{references}

\end{document}